\documentclass[11 pt]{article}
\usepackage{graphicx}
\usepackage{amsmath}
\usepackage{amssymb}
\usepackage{amsthm}
\usepackage{enumerate}
\usepackage{color}
\usepackage{mathdots}
\usepackage{sectsty}
\usepackage[hidelinks]{hyperref}
\usepackage{tikz}
\usepackage{caption}
\usepackage{adjustbox}
\usepackage{fancyhdr}
\usepackage{verbatim}
\usepackage{hyperref}
\usepackage{url}

\sectionfont{\scshape\centering\fontsize{11}{14}\selectfont}
\subsectionfont{\scshape\fontsize{11}{14}\selectfont}

\newcommand\shorttitle{Partitions into Beatty sequences}
\newcommand\authors{\small Nian Hong Zhou}

\hypersetup{
    colorlinks=true,       
    linkcolor=blue,          
    citecolor=cyan,        
}

\theoremstyle{plain}
\newtheorem{theorem}{Theorem}[section]
\newtheorem{lemma}[theorem]{Lemma}
\newtheorem{corollary}[theorem]{Corollary}
\newtheorem{proposition}[theorem]{Proposition}

\theoremstyle{remark}
\newtheorem{remark}{Remark}[section]

\makeatletter

\newcommand{\Rmnum}[1]{\expandafter\@slowromancap\romannumeral #1@}

\def\qb{\mathbb Q}
\def\rb{\mathbb R}
\def\nb{\mathbb N}
\def\zb{\mathbb Z}

\def\rrw{\rightarrow}

\numberwithin{equation}{section}

\title{Partitions into Beatty sequences} 

\author{Nian Hong \textsc{Zhou}} 

\date{\today} 

\begin{document}

\maketitle
\begin{abstract}
Let $\alpha>1$ be an irrational number. We establish asymptotic formulas for the number of partitions of $n$ into summands and distinct summands, chosen from the \emph{Beatty sequence} $(\lfloor\alpha m\rfloor)_{m\in\nb}$. This improves some results of Erd\"{o}s and Richmond established in 1977.
\end{abstract}
\section{Introduction and statement of results}
A partition of an integer $n$ is a sequence of non-increasing positive integers whose sum equals $n$. The study of the asymptotic behavior of various types of integer partition has a long history, see Hardy and Ramanujan \cite{HR1918}, Ingham \cite{MR5522}, Roth and Szekeres \cite{MR0067913}, Meinardus \cite{MR62781} and Richmond \cite{MR384731} for example. One of the most celebrated results is the asymptotic formula for $p(n)$, the number of unrestricted partitions of $n$. Hardy and Ramanujan \cite{HR1918} proved
\begin{align}\label{eqhr}
p(n)\sim \frac{1}{4\sqrt{3}n}e^{\pi\sqrt{2n/3}},
\end{align}
as integer $n\rrw+\infty$. In the same paper, they also established an asymptotic formula for the number $q(n)$ of partitions of $n$ with unequal parts. As $n\rrw \infty$
\begin{align*}
q(n)\sim \frac{1}{4\cdot 3^{1/4}\cdot n^{3/4}}e^{\pi\sqrt{n/3}}.
\end{align*}

\medskip

Let $\alpha>1$ be an irrational number and $\lfloor x\rfloor$ denotes the largest integer $\le x$. As a natural extension which has not previously appeared in the literature, Erd\"{o}s and Richmond \cite{MR535018} investigate the asymptotic behavior of $p_{\alpha}(n)$, and $q_{\alpha}(n)$, the number of partitions of $n$ into summands, and distinct summands, respectively, chosen
from the \emph{Beatty sequence} $(\lfloor\alpha m\rfloor)_{m\in\nb}$.  They gave asymptotic formulas with an error term of $p_{\alpha}(n)$ and $q_{\alpha}(n)$ for almost all $\alpha$ (in the Lebesgue sense).

\medskip

To introduce the main results of Erd\"{o}s and Richmond \cite{MR535018} conveniently, we introduce the definition of \emph{irrationality measure}. Let $\|x\|:=\min_{n\in\zb}|x-n|$ and $\alpha\in\rb$. Recall that $\mu\in\rb$ is an irrationality exponent of $\alpha$, if
$$0<q^{-1}\|q\alpha\|<  q^{-\mu},$$
has (at most) finitely many solutions $q\in\nb$. We denote by $\mu(\alpha)$ the infimum of such irrationality exponents $\mu$ and call it the irrationality measure of $\alpha$. If $\mu(\alpha)=\infty$ then we call $\alpha$ a \emph{Liouville number}. We note that the irrationality measure of an irrational number is always $\ge 2$.

\medskip

One of the critical issues in \cite{MR535018} is the convergence of the following Dirichlet series
\begin{equation}\label{seriesj}
J_{\alpha}(s):=\sum_{\ell\ge 1}\frac{\widetilde{B}_1(\alpha\ell)}{\ell^s},
\end{equation}
where $\widetilde{B}_1(x)=\{x\}-1/2$, and $\{x\}:=x-\lfloor x\rfloor$ is the fractional part of $x$. Hardy and Littlewood \cite[p. 248, (a)]{MR3069402} proved that if $\mu(\alpha)\in[2,\infty)$ then $J_{\alpha}(s)$ is convergent for $s>1-1/(\mu(\alpha)-1)$
\footnote{The condition stated in \cite[pp. 213, Equation (1.331)]{MR3069402} is
$$ n^{h}|\sin(\alpha n\pi)|\ge A>0$$
for all $n\in\nb$. This is equivalent to our definition for irrationality exponent when letting $\mu=1+h$.
}.
Thanks to this result, Erd\"{o}s and Richmond \cite{MR535018} established asymptotic formulas with error term of $p_{\alpha}(n)$ and $q_{\alpha}(n)$ when irrational number $\alpha>1$ has a finite irrationality measure $\mu(\alpha)$, see \cite[Theorem 2]{MR535018}\footnote{The condition stated in \cite{MR535018} is there exist $\lambda\in\rb$ such that
\begin{equation*}
 |\ell^{1+\lambda+\varepsilon}\sin(\alpha\ell\pi)|\rrw \infty,
\end{equation*}
holds for any $\varepsilon>0$, as integer $\ell\rrw\infty$. This is equivalent to our definition for irrationality exponent when letting $\mu=2+\lambda$.}' \footnote{Note that there exist serval typos in the statement of \cite[Theorem 2]{MR535018} as well as its proof. For the corrected leading asymptotic formulas of $p_{\alpha}(n)$ and $q_{\alpha}(n)$ with $\mu(\alpha)<\infty$, see Theorem \ref{mth} of this paper.}. However, for $\alpha>1$ being a Liouville number, that is $\mu(\alpha)=\infty$, they only can prove
$$\log q_{\alpha}(n)=\pi\sqrt{\frac{n}{3\alpha}}+O(n^{\varepsilon})$$
and
$$\log p_{\alpha}(n)=\pi\sqrt{\frac{2n}{3\alpha}}+O(n^{\varepsilon}),$$
for any $\varepsilon>0$.

\medskip

In this paper, we are interesting the asymptotic formulas of $p_{\alpha}(n)$ and $q_{\alpha}(n)$, when $\alpha>1$ being a Liouville number. The main result of this paper is the following Theorem \ref{mth}.

\begin{theorem}\label{mth}Let $\alpha>1$ be an irrational number. As $n\rrw\infty$
$$
q_{\alpha}(n)\sim \frac{\exp\left(\pi\sqrt{\frac{n}{3\alpha}}\right)}{2^{2-1/2\alpha}(3\alpha)^{1/4}n^{3/4}},
$$
and
\begin{equation}\label{eqm10}
p_{\alpha}(n)=\frac{
\exp\left(2\pi\sqrt{\frac{n}{6\alpha}}\right)}{n^{1-{1}/{4\alpha}+o(1)}}.
\end{equation}
Furthermore, if the series \eqref{seriesj} of $J_{\alpha}(1)$ is convergent, then
\begin{equation}\label{eqm}
p_{\alpha}(n)\sim \frac{\exp\left(2\pi\sqrt{\frac{n}{6\alpha}}\right)}{\Lambda_{\alpha} n^{1-1/4\alpha}},
\end{equation}
wherewith $\gamma$ denote the Euler-Mascheroni constant,
\begin{equation}\label{eqlm}
\Lambda_{\alpha}=4\sqrt{3}(\pi e^{-\gamma})^{\frac{1}{2\alpha}}\left(\frac{\alpha}{6}\right)^{1/4\alpha}\prod_{\ell\ge 1}\left(1-\frac{\{\alpha \ell\}}{\alpha \ell}\right)e^{\frac{1}{2\alpha \ell}}.
\end{equation}
\end{theorem}
\begin{remark}Theorem \ref{mth} is new for $\alpha>1$ being a Liouville number. Surprisingly, for \emph{all irrational numbers $\alpha>1$}, we can find an asymptotic formula of type \eqref{eqhr} for $q_{\alpha}(n)$.
However, for $p_{\alpha}(n)$ if there is no additional assumption, such as the convergence of the series \eqref{seriesj} of $J_{\alpha}(1)$, we can only get \eqref{eqm10} at present.
\end{remark}

\medskip

By using a result of Ostrowski \cite{MR3069389} on Diophantine approximation, we give an effective condition for the validity of the asymptotic formula \eqref{eqm}.

\begin{corollary}\label{mcor}If there exists a function $\psi:\rb_+\rrw \rb_+$ of which is decreasing for all sufficiently large $x$ such that
\begin{equation*}
\int_{1}^{\infty}\frac{\psi(x)\,dx}{x}<+\infty,
\end{equation*}
as well as for all integer $q\ge 1$,
\begin{equation*}
q^{-1}\|q\alpha\|> e^{-q\psi(q)}.
\end{equation*}
Then the series \eqref{seriesj} of $J_{\alpha}(1)$ is convergent. Further, \eqref{eqm} in Theorem \ref{mth} holds.
\end{corollary}
\begin{remark}
It might be an interesting problem that whether the series \eqref{seriesj} of $J_{\alpha}(1)$ is convergent for all irrational numbers $\alpha$.
\end{remark}

\begin{remark}Taking $\psi(x)=x(\log x+1)^{-1-\delta},(\delta>0)$, then $\psi(x)$ satisfies the conditions of Corollary \ref{mcor}.
If $\alpha\in \rb\setminus \qb$ has a finite irrationality measure, that is there exist a $\lambda\ge 2$ such that
$$q^{-1}\|q\alpha\|\gg q^{-\lambda},$$
for all integer $q\ge 1$. Then, since $q^{-\lambda}\gg e^{-\psi(q)}$,
we have $\alpha>1$ satisfies the conditions of Corollary \ref{mcor}. Therefore, we can see from Wolfram Web Resources \cite{irm} that $\alpha$ can take $1+\log 2$, $\log 3$, $e$, $\pi$, $\pi^2$, $\zeta(3)$, all irrational algebraic numbers greater than $1$ and so on.

\end{remark}

We shall give some numerical data to support Theorem \ref{mth} and Corollary \ref{mcor}. Let us consider the cases of $\alpha=\sqrt{2}$. One can prove in Appendix \ref{A} that $$\Lambda_{\sqrt{2}}\in(5.7331, 5.7339).$$ Denoting by
$$\hat{q}_{\sqrt{2}}(n)= n^{-3/4}e^{\pi\sqrt{\frac{n}{3\sqrt{2}}}}$$
and
$$\hat{p}_{\sqrt{2}}(n)=n^{-1+1/4\sqrt{2}}e^{\pi\sqrt{\frac{2n}{3\sqrt{2}}}}.$$
Then, with the help of {\bf Mathematica}, the use of Theorem \ref{mth} gives
\begin{equation}\label{eq1}
\hat{q}_{\sqrt{2}}(n)\sim (4.493\cdots)q_{\sqrt{2}}(n),\;\;\mbox{and}\;\; \hat{p}_{\sqrt{2}}(n)\sim (5.773\cdots)p_{\sqrt{2}}(n).
\end{equation}
We illustrate some of our results in the following(All computations are done in {\bf Mathematica}).
\begin{center}
\captionof{table}{Numerical data for $q_{\sqrt{2}}(n)$.}
  \begin{tabular}{ c | c | c | c  }
\hline
    $n$ & $q_{\sqrt{2}}(n)$ & $\hat{q}_{\sqrt{2}}(n)$ & $\hat{q}_{\sqrt{2}}(n)/q_{\sqrt{2}}(n)$\\
\hline
   $50 $ & $552$ &$2568.04$ & $\sim 4.65225$ \\
   $100$ & $28870$ & $133026.0$ & $\sim 4.60774$ \\
   $200$ & $9582142$ & $4.38416\cdot 10^7$ & $\sim 4.57534$ \\
   $400$ & $43472367216$ & $1.97845\cdot 10^{11}$ & $\sim 4.55105$ \\
   $800$ & $7972258288121676$ & $3.61410\cdot 10^{16}$ & $\sim 4.53334$ \\
   $1600$ & $273790497114888860128182$ & $1.23780\cdot 10^{24}$ & $\sim 4.52097$ \\
\hline
  \end{tabular}
\end{center}

\begin{center}
\captionof{table}{Numerical data for $p_{\sqrt{2}}(n)$.}
  \begin{tabular}{ c | c | c | c  }
\hline
    $n$ & $p_{\sqrt{2}}(n)$ & $\hat{p}_{\sqrt{2}}(n)$ & $\hat{p}_{\sqrt{2}}(n)/p_{\sqrt{2}}(n)$\\
\hline
   $25 $ & $560$ &$3412.03$ & $\sim 6.09291$ \\
   $50 $ & $28086$ &$167998.0$ & $\sim 5.98154$ \\
   $100$ & $8892735$ & $5.26274\cdot10^7$ & $\sim 5.91802$ \\
   $200$ & $38427241214$ & $2.25740\cdot 10^{11}$ & $\sim 5.87448$ \\
   $400$ & $6706078262183805$ & $3.91959\cdot 10^{16}$ & $\sim 5.84483$ \\
   $800$ & $219091729965354807601257$ & $1.27599\cdot 10^{24}$ & $\sim 5.82401$ \\
\hline
  \end{tabular}
\end{center}
Comparing the above two tables with \eqref{eq1}, we see that the numerical data supports our main result Theorem \ref{mth} and Corollary \ref{mcor}.

\section{The proof of the  main results}
In the section, we prove Theorem \ref{mth} and Corollary \ref{mcor}.
Let $\alpha>1$ be an irrational number and $t>0$. Denoting by
$$L_{\alpha}(t)=-\sum_{\ell\ge 1}\log\left(1-e^{-t\lfloor \alpha\ell\rfloor}\right).$$
The following Proposition \ref{pro21} follows from Erd\"{o}s and Richmond \cite[Theorem 1]{MR535018}.
\begin{proposition}\label{pro21} As $n\rrw +\infty$
\begin{equation*}
q_{\alpha}(n)\sim \frac{\exp\left(L_{\alpha}(x)-L_{\alpha}(2x)+nx\right)}{\sqrt{2\pi (L_{\alpha}''(x)-4L_{\alpha}''(2x))}},
\end{equation*}
where $x\in\rb_+$ solves the equation: $
L_{\alpha}'(x)-2L_{\alpha}'(2x)+n=0.$\\
Furthermore, as $n\rrw +\infty$
\begin{equation*}
p_{\alpha}(n)\sim \frac{\exp\left(L_{\alpha}(y)+ny\right)}{\sqrt{2\pi L_{\alpha}''(y)}},
\end{equation*}
where $y\in\rb_+$ solves the equation: $
L_{\alpha}'(y)+n=0.$
\end{proposition}
We see from Erd\"{o}s and Richmond \cite{MR535018} that Proposition \ref{pro21} follows from the work of Roth and Szekeres \cite{MR0067913}, as well as the equidistribution properties of the sequence $(\lfloor\alpha m\rfloor)_{m\in\nb}$. Under above Proposition \ref{pro21}, Theorem \ref{mth} will follows from the following proposition.
\begin{proposition}\label{pro22}Let $t\rrw 0^+$.
For any irrational $\alpha>1$ we have
$$L_{\alpha}(t)=\frac{\pi^2}{6\alpha t}+\frac{1-\alpha^{-1}}{2}\log t+ o(|\log t|),$$
$$L_{\alpha}(t)-L_{\alpha}(2t)=\frac{\pi^2}{12\alpha t}-\frac{1-\alpha^{-1}}{2}\log 2 + o(1),$$
$$L_{\alpha}'(t)=-\frac{\pi^2}{6\alpha t^2}+\frac{1-\alpha^{-1}}{2t} + o\left(\frac{1}{t}\right),$$
and
$$L_{\alpha}''(t)=\frac{\pi^2}{3\alpha t^3}+ O\left(\frac{1}{t^2}\right).$$
Furthermore, if the series \eqref{seriesj} of $J_{\alpha}(1)$ is convergent then
$$L_{\alpha}(t)=\frac{\pi^2}{6\alpha t}+\frac{1-\alpha^{-1}}{2}\log t+c_{\alpha} + o(1),$$
where
$$c_{\alpha}=\frac{\gamma}{2\alpha}-\frac{1}{2}\log\left(2\pi\right)+\frac{1-\alpha^{-1}}{2}\log \alpha+\sum_{\ell\ge 1}\left(\frac{1}{2\alpha\ell}+\log\left(1-\frac{\{\alpha\ell\}}{\alpha\ell}\right)\right).$$
\end{proposition}
This proposition is a direct consequence of Proposition \ref{pro31}, Lemma \ref{lem32}, and Lemma \ref{lem34} of Section \ref{sec3}.

\medskip

We now prove Theorem \ref{mth}. We just give the proof for $p_{\alpha}(n)$, the proof for $q_{\alpha}(n)$ is similar. Using Proposition \ref{pro21} and Proposition \ref{pro22}, we find that
\begin{align*}
n+\frac{\pi^2}{6\alpha y^2}-\frac{1-\alpha^{-1}}{2y}\left(1+o(1)\right)=0,
\end{align*}
as $n\rrw\infty$. This immediately implies
\begin{equation}\label{eqy}
y=\frac{\pi}{\sqrt{6\alpha n}}-\frac{1-\alpha^{-1}}{4 n}\left(1+o(1)\right),
\end{equation}
as $n\rrw \infty$. Substituting \eqref{eqy} into Proposition \ref{pro21} and Proposition \ref{pro22}, by simplification we obtain the proof of Theorem \ref{mth} for $p_{\alpha}(n)$.

\medskip

We now give the sketch of the proof Corollary \ref{mcor}. Clearly, we just need to prove that under the conditions of Corollary \ref{mcor}, the series \eqref{seriesj} of $J_{\alpha}(1)$ is convergent. Using integration by parts for Riemann-Stieltjes integrals, the convergence of the series \eqref{seriesj} easily follows the following.
\begin{proposition}Let $\alpha$ be satisfies the conditions of Corollary \ref{mcor}. Then, there exist a constant $c_0>0$ such that
$$\sum_{1\le \ell\le x}\widetilde{B}_1(\alpha\ell)\ll \frac{x\psi(c_0\log x)}{\log x},$$
for all sufficiently large $x$.
\end{proposition}

\begin{proof}Since for any integer $q\ge 1$,
$$q^{-1}\|q\alpha\|> e^{-q\psi(q)}$$
means that for any $p\in\zb$ with $\gcd(p,q)=1$, there exist a constant $C>0$ such that
$$\left|\alpha-\frac{p}{q}\right|>Ce^{-q\psi(q)}.$$
Thus it not difficult to prove that there exists a constant $c_0>0$ such that
\begin{align}\label{eqm1}
S(x)=\sum_{1\le \ell\le x}\widetilde{B}_1(\alpha\ell)\ll  \frac{x\psi(c_0\log x)}{\log x},
\end{align}
for all sufficiently large $x$, by use of the same idea for the cases of $\alpha$ such that
$$\left|\alpha-\frac{p}{q}\right|>e^{-\lambda q}, ~(\lambda>0~\text{is a constant}),$$
in Ostrowski \cite[p.83]{MR3069389}. This completes the proof of the proposition.
\end{proof}
\section{The proof of Proposition \ref{pro22}}\label{sec3}
The following proposition gives a very well decomposition of the logarithm of the generating function $L_{\alpha}(t)$. From which we can find the main contribution of the asymptotics of the generating function.
\begin{proposition}\label{pro31}Let $t>0$. We have
\begin{align*}
L_{\alpha}(t)=L_{1}(\alpha t)+2^{-1}tD(\alpha t)+R_{\alpha}(t)+E_{\alpha}(t),
\end{align*}
where
\begin{align*}
D(t)=\sum_{n\ge 1}\frac{1}{e^{nt}-1},\;
\;R_{\alpha}(t)=\sum_{\ell\ge 1}\frac{ t\widetilde{B}_1(\alpha\ell)}{e^{t\alpha\ell}-1},
\end{align*}
and
\begin{align*}
E_{\alpha}(t)=\sum_{\ell\ge 1}\int_0^{\{\alpha\ell\}}\,du\int_{0}^{u}K\left(\frac{(\alpha\ell-v)t}{2}\right)\frac{\,d v}{(\alpha\ell-v)^2}
\end{align*}
with $K(u)=u^2/\sinh^2(u)$.

\end{proposition}
\begin{proof}First of all, by a direct calculation, we find that
\begin{align*}
L_{\alpha}(t)&=-\sum_{\ell\ge 1}\log\left(1-e^{-t \alpha\ell}\right)+\sum_{\ell\ge 1}\int_{t\lfloor \alpha\ell\rfloor}^{t \alpha\ell}\,d\log\left(1-e^{-u}\right)\\
&=L_{1}(\alpha t)+\sum_{\ell\ge 1}\int_0^{\{\alpha\ell\}}\frac{t\,du}{e^{t (\alpha\ell-u)}-1}.
\end{align*}
For the second sum above, we split that
\begin{align*}
\sum_{\ell\ge 1}\int_0^{\{\alpha\ell\}}\frac{t\,du}{e^{t (\alpha\ell-u)}-1}=&\sum_{\ell\ge 1}\frac{t\{\alpha\ell\}}{e^{t \alpha\ell}-1}+\sum_{\ell\ge 1}\int_0^{\{\alpha\ell\}}\,du\int_{0}^{u}\frac{t^2e^{t (\alpha\ell-v)}}{(e^{t (\alpha\ell-v)}-1)^2}\,d v\\
&=\frac{t}{2}\sum_{\ell\ge 1}\frac{1}{e^{t \alpha\ell}-1}+\sum_{\ell\ge 1}\frac{ t\widetilde{B }_1(\alpha\ell)}{e^{t\alpha\ell}-1}+E_{\alpha}(t),
\end{align*}
where $\widetilde{B }_1(x)=\{x\}-1/2$, and
\begin{align*}
E_{\alpha}(t)=\sum_{\ell\ge 1}\int_0^{\{\alpha\ell\}}\,du\int_{0}^{u}\left[\frac{(\alpha\ell-v)t/2}{\sinh((\alpha\ell-v)t/2)}\right]^2\frac{\,d v}{(\alpha\ell-v)^2}.
\end{align*}
This completes the proof of the proposition.
\end{proof}

In the following content, we estimate each component of the above decomposition of the logarithm of the generating function $L_{\alpha}(t)$.

\begin{lemma}\label{lem32}Let $t\rrw 0^+$. Denoting by
$$\hat{L}_{\alpha}(t)=L_1(\alpha t)+2^{-1}tD(\alpha t),$$
then we have
$$\hat{L}_{\alpha}(t)=\frac{\pi^2}{6\alpha t}+\frac{1-\alpha^{-1}}{2}\log t+\frac{1-\alpha^{-1}}{2}\log \alpha+\frac{\gamma}{2\alpha}-\frac{1}{2}\log\left(2\pi\right)+O(t^{1/2}),$$
$$\hat{L}_{\alpha}'(t)=-\frac{\pi^2}{6\alpha t^2}+\frac{1-\alpha^{-1}}{2t}+O(t^{-1/2}),$$
and
$$\hat{L}_{\alpha}''(t)=\frac{\pi^2}{3\alpha t^3}+O\left(\frac{1}{t^{2}}\right).$$
\end{lemma}

\begin{proof}Let $t\rrw 0^+$. Notice that
$$D(t)=\sum_{n\ge 1}\frac{e^{-nt}}{1-e^{-nt}}=\sum_{n\ge 1}\tau(n)e^{-nt},$$
where $\tau(n)=\sum_{d|n}1$ is the  divisor function. Using the well-known fact for divisor function $\tau(n)$ that
$$\sum_{n\le x}\tau(n)=(\log x+2\gamma-1)x+O(x^{1/2}),$$
and integration by parts for Riemann-Stieltjes integrals, we have
\begin{align*}
t\sum_{n\ge 1}n^j\tau(n)e^{-nt}&=t\int_{1-}^{\infty}x^je^{-xt}\,d\left(\sum_{n\le x}\tau(n)\right)\\
&=t\int_{1}^{\infty}x^je^{-xt}\,d\left((\log x+2\gamma-1)x\right)+O(t^{1/2-j})\\
&=t^{-j}\int_{t}^{\infty}x^j(\log x+2\gamma-\log t)e^{-x}\,dx+O(t^{1/2})\\
&=j!t^{-j}(\gamma-\log t)+t^{-j}\int_{0}^{\infty}x^j(\log x+\gamma)e^{-x}\,dx+O(t^{1/2-j}).
\end{align*}
Further, by note that
$$\int_{0}^{\infty}(\log x+\gamma)e^{-x}\,dx=0\; \text{and}\; \int_{0}^{\infty}x(\log x+\gamma)e^{-x}\,dx=1,$$
which immediately implies
$$
tD(t)=\gamma-\log t+O(t^{1/2}),\;
(tD(t))'=-t^{-1}+O(t^{-1/2}),\; \mbox{and}\;
(tD(t))''\ll t^{-2},
$$
by a direct calculation.  Finally, together with the well-known transform relation that
$$L_1(t)=-\frac{t}{24}-\frac{1}{2}\log(2\pi)+\frac{1}{2}\log t+\frac{\pi^2}{6t}+L_1\left(\frac{4\pi^2}{t}\right),$$
see for example \cite[Equation (1.42)]{HR1918}, and above estimates for $D(t)$ we immediately obtain the proof of the lemma.
\end{proof}

\begin{lemma}\label{lem33}Let $t\rrw 0^+$. We have
$$E_{\alpha}(t)=E_{\alpha}(0)+o(1)$$
with
$$E_{\alpha}(0)=-\sum_{\ell\ge 1}\left(\frac{\{\alpha\ell\}}{\alpha\ell}+\log\left(1-\frac{\{\alpha\ell\}}{\alpha\ell}\right)\right),$$
and for each integer $k\ge 1$,
$$E_{\alpha}^{\langle k\rangle}(t)\ll_k t^{1-k}.$$
\end{lemma}
\begin{proof}The proof for the value of $E_{\alpha}(0)$ is direct by use of Lebesgue's Dominated Convergence Theorem.  We now prove the estimates for the derivative of $E_{\alpha}(t)$. Using the definition of $E_{\alpha}(t)$, we have
\begin{align*}
E_{\alpha}^{\langle k\rangle}(t)=\sum_{\ell\ge 1}\int_0^{\{\alpha\ell\}}\,du\int_{0}^{u}K^{\langle k\rangle}\left(\frac{(\alpha\ell-v)t}{2}\right)\frac{\,d v}{2^k(\alpha\ell-v)^{2-k}}.
\end{align*}
Note that
$$K^{\langle k\rangle}(u)\ll_k \min(1, e^{-u}),$$
we have
\begin{align*}
E_{\alpha}^{\langle k\rangle}(t)\ll_k\sum_{1\le \ell\le 1/t}\ell^{k-2}+\sum_{\ell> 1/t}e^{-\alpha \ell t/2}\ell^{k-2}\ll_k t^{1-k},
\end{align*}
as $t\rrw 0^+$. This completes the proof of lemma.
\end{proof}
We finally prove the following lemma which plays an important role in this paper.
\begin{lemma}\label{lem34}As $t\rrw 0^+$
$$R_{\alpha}'(t)=o(t^{-1})\;\text{and}\; R_{\alpha}''(t)\ll t^{-2}.$$
Furthermore,
$$
R_{\alpha}(t)=o(|\log t|)\;\;\mbox{and}\;\;R_{\alpha}(t)-R_{\alpha}(2t)=o(1).
$$
Moreover, if the series \eqref{seriesj} of $J_{\alpha}(1)$ is convergent then
$$R_{\alpha}(t)=\alpha^{-1}J_{\alpha}(1)+o(1).$$
\end{lemma}
\begin{proof}Let $t\rrw 0^+$. The proof of the estimate for $R_{\alpha}''(t)$ is a direct calculation. In fact, by the definition of $R_{\alpha}(t)$,
\begin{align*}
R_{\alpha}''(t)&=\frac{\,d^2}{\,d t^2}\sum_{\ell\ge 1}\frac{t\widetilde{B}_{1}(\alpha\ell)}{e^{\alpha\ell t}-1}\\
&=\sum_{\ell\ge 1}\widetilde{B}_{1}(\alpha\ell)(\alpha\ell)\frac{\,d^2}{\,d u^2}\bigg|_{u=\alpha\ell t}\frac{u}{e^u-1}\\
&\ll \sum_{1\le \ell\le 1/t}\ell+\sum_{\ell>1/t}\ell(\ell t)e^{-\alpha \ell t}\ll t^{-2}.
\end{align*}
We now give the proof of the estimate for $R_{\alpha}'(t)$. Since the sequence $(\widetilde{B}_1(\alpha\ell))_{\ell\in\nb}$ is equidistributed in [-1/2,1/2) for all $\alpha\in \rb\setminus \qb$, we have
\begin{align}\label{eqd1}
S_{\alpha}(x):=\sum_{1\le \ell\le x}\widetilde{B}_1(\alpha\ell)=o(x),
\end{align}
as $x\rrw+\infty$. Using the definition of $R_{\alpha}(t)$, we have
\begin{align*}
R_{\alpha}'(t)&=\sum_{\ell\ge 1}\frac{e^{\alpha\ell t}-1-\alpha\ell te^{\alpha \ell t}}{(e^{\alpha\ell t}-1)^2}\widetilde{B}_1(\alpha\ell)\\
&=\int_{1-}^{\infty}\frac{e^{\alpha x t}-1-\alpha x e^{\alpha x t}}{(e^{\alpha x t}-1)^2}\,d S_{\alpha}(x)\\
&=\int_{\alpha t}^{\infty} S_{\alpha}(u/\alpha t) \left(\frac{e^{u}-1-ue^{u}}{(e^{u}-1)^2}\right)'\,du.
\end{align*}
Thus the using of \eqref{eqd1} implies that
\begin{align*}
R_{\alpha}'(t)&\ll\int_{\alpha t}^{(\alpha t)^{1/2}} \left|\frac{u}{\alpha t}\right| \,du+o\left(\frac{1}{\alpha t}\int_{(\alpha t)^{1/2}}^{\infty} u\left|\left(\frac{e^{u}-1-ue^{u}}{(e^{u}-1)^2}\right)'\right|\,du\right)\\
&\ll 1+o(t^{-1})=o(t^{-1}).
\end{align*}
Therefore,
\begin{align*}
R_{\alpha}(t)&=R_{\alpha}(1)+\int_{1}^{t}R_{\alpha}'(u)\,du\\
&\ll 1+\int_{1/\log(1/t)}^{1}\frac{\,du}{u}+o\left(\int_{t}^{1/\log(1/t)}\frac{\,du}{u}\right)=o\left(|\log t|\right),
\end{align*}
and
\begin{align*}
R_{\alpha}(t)-R_{\alpha}(2t)&=t\int_{2}^{1}R_{\alpha}'(tu)\,du\ll t\times o\left(\int_{1}^{2}\frac{1}{tu}\,du\right)=o(1).
\end{align*}
Moreover, if the series \eqref{seriesj} of $J_{\alpha}(1)$ is convergent, then the use of integration by parts for Riemann-Stieltjes integrals implies
\begin{align*}
R_{\alpha}(t)-\alpha^{-1}J_{\alpha}(1)=&\sum_{\ell\ge 1}\left(\frac{\alpha\ell t}{e^{\alpha\ell t}-1}-1\right)\frac{\widetilde{B}_{1}(\alpha\ell)}{\alpha\ell}\\
\ll & \sum_{1\le \ell\le t^{-1/2}}|\alpha\ell t|\frac{1}{\alpha\ell}+\left|\int_{t^{-1/2}}^{\infty}\left(\frac{\alpha x t}{e^{\alpha x t}-1}-1\right)\,d \left(\sum_{1\le \ell\le x}\frac{\widetilde{B}_{1}(\alpha\ell)}{\alpha\ell}\right)\right|\\
= & \sum_{1\le \ell\le t^{-1/2}}t+\left|\int_{t^{-1/2}}^{\infty}\left(\frac{\alpha x t}{e^{\alpha x t}-1}-1\right)\,d \left(\alpha^{-1}J_\alpha(1)+o(1)\right)\right|\\
\ll& t^{1/2}+o(1)\left(t^{1/2}+\int_{t^{-1/2}}^{\infty}\left|\frac{\,d}{\,dx}\frac{\alpha x t}{e^{\alpha x t}-1}\right|\,d x\right)=o(1),
\end{align*}
which completes the proof of the lemma.
\end{proof}

\appendix
\section{Numerical approximation for $\Lambda_{\alpha}$}\label{A}
In this appendix we investigate the  numerical approximation for $\Lambda_{\alpha}$. Denoting by
\begin{equation}\label{A1}
\Pi_{\alpha}=\prod_{\ell\ge 1}\left(1-\frac{\{\alpha\ell\}}{\alpha\ell}\right)e^{\frac{1}{2\alpha\ell}},
\end{equation}
then using \eqref{eqlm} we see that $\Lambda_{\alpha}=4\sqrt{3}(\pi e^{-\gamma})^{1/2\alpha}(\alpha/6)^{1/4\alpha}\Pi_{\alpha}$.
Taking a logarithm of \eqref{A1} we obtain
\begin{align}\label{A2}
\log \Pi_{\alpha}&=\sum_{\ell\ge 1}\left(\frac{1}{2\alpha \ell}-\log\left(1-\frac{\{\alpha\ell\}}{\alpha\ell}\right)\right)\nonumber\\
&=\bigg(\sum_{1\le \ell\le N}+\sum_{\ell>N}\bigg)\left(\frac{1}{2\alpha \ell}-\log\left(1-\frac{\{\alpha\ell\}}{\alpha\ell}\right)\right)\nonumber\\
&=:\Sigma_m(N)+\Sigma_e(N),
\end{align}
where $N>10$ is an integer will be chosen for give a good numerical approximation for $\Lambda_{\alpha}$.

\medskip

We now bound the error term $\Sigma_e(N)$. We rewritten the sum of $\Sigma_e(N)$ as
\begin{align}\label{A3}
\Sigma_{e}(N)=&-\sum_{\ell>N}\frac{\widetilde{B}_1(\alpha\ell)}{\alpha\ell}-\sum_{\ell>N}\left(\log\left(1-\frac{\{\alpha\ell\}}{\alpha\ell}\right)+\frac{\{\alpha\ell\}}{\alpha\ell}\right)\nonumber\\
=&-\Sigma_{1e}(N)+\Sigma_{e2}(N).
\end{align}
It is not difficult to give a bound for the second sum above that
\begin{align}\label{eqsge2}
0<\Sigma_{e2}(N)&< -\sum_{\ell>N}\left(\log\left(1-\frac{1}{\alpha\ell}\right)+\frac{1}{\alpha\ell}\right)\nonumber\\
&<-\int_{N}^{\infty}\left(\log\left(1-\frac{1}{\alpha x}\right)+\frac{1}{\alpha x}\right)\,dx<\frac{3\alpha N-2}{6\alpha N(\alpha N-1)}.
\end{align}
Using part integration to $\Sigma_{e1}(N)$ we have
\begin{align}\label{eqsge1}
\Sigma_{e1}(N)&=-\int_{N}^{\infty}\frac{1}{\alpha x}\,d S_{\alpha}(x)\nonumber\\
&=\frac{S_{\alpha}(N)}{\alpha N}-\frac{1}{\alpha}\int_{N}^{\infty}\frac{S_{\alpha}(x)}{x^2}\,dx.
\end{align}

We now focus on the approximation of $\Lambda_{\alpha}$ to a class of irrational numbers $\alpha$ in which the
partial quotients of the continued fraction expansion of $\alpha$ are bounded. In other world, $\alpha$ has the following continued fraction expansion
$$\alpha=[a_0; a_1, a_2,\ldots]=a_0+\cfrac{1}{a_1+\cfrac{1}{a_2+\cdots}}$$
with all $a_j\le A$ for some $A>0$. In this cases Ostrowski \cite[pp. 80--81]{MR3069389} proved that
\begin{align}\label{eqS}
|S_{\alpha}(x)|\le \frac{3}{2}A\log x,
\end{align}
for all $x>10$. Substituting \eqref{eqS} in \eqref{eqsge1} we find that
$$|\Sigma_{e1}(N)|\le \frac{3A\log N}{\alpha N}+\frac{3A}{2\alpha N}.$$
Combining \eqref{A2}--\eqref{eqsge2}, and above we obtain
\begin{align}\label{A7}
\left|\log \Pi_{\alpha}-\Sigma_m(N)\right|< \frac{3A}{\alpha N}\left(\log N+\frac{1}{2}\right)+\frac{3\alpha N-2}{6\alpha N(\alpha N-1)}.
\end{align}

We now give the numerical approximation for $\Lambda_{\sqrt{2}}$. Note that $\sqrt{2}=[1;2,2,2,\ldots]$, that is the partial quotients of the continued fraction expansion are bounded by $2$. Hence from \eqref{A7} we have
\begin{align*}
\left|\log \Pi_{\sqrt{2}}-\Sigma_m(N)\right|< \frac{3\sqrt{2}}{N}\left(\log N+\frac{1}{2}\right)+\frac{3N-\sqrt{2}}{6N(\sqrt{2} N-1)}.
\end{align*}
Taking $N=10^6$, then using {\bf Mathematica} we find that $$\log \Pi_{\sqrt{2}}=-0.127496+ 6.11\times 10^{-5}\theta,$$
for some $\theta\in(-1,1)$. Hence
$$\Lambda_{\sqrt{2}}=4\sqrt{3}(\pi e^{-\gamma})^{1/2\sqrt{2}}(\sqrt{2}/6)^{1/4\sqrt{2}}\Pi_{\sqrt{2}}\in(5.7731, 5.7739).$$



\bigskip
\noindent
{\sc Nian Hong Zhou\\
School of Mathematics and Statistics, Guangxi Normal University\\
No.1 Yanzhong Road, Yanshan District, Guilin, 541006\\
Guangxi, PR China}\newline
Email:~\href{mailto:nianhongzhou@outlook.com; nianhongzhou@gxnu.edu.cn}{\small nianhongzhou@outlook.com; nianhongzhou@gxnu.edu.cn}

\end{document}